\documentclass[noamsfonts,a4paper,10pt]{amsart}
\usepackage[margin=2cm]{geometry}
\usepackage[bitstream-charter,cal]{mathdesign}
\usepackage{hyperref}
\usepackage{setspace}
\theoremstyle{plain}
\newtheorem{theorem}{Theorem}[section]
\newtheorem*{theorem*}{Theorem}
\newtheorem{lemma}[theorem]{Lemma}
\newtheorem{corollary}[theorem]{Corollary}
\newtheorem*{corollary*}{Corollary}
\theoremstyle{definition}
\newtheorem{definition}[theorem]{Definition}

\title[Positive Hahn-Banach separation theorems]{A solution to Haagerup's problem and positive Hahn-Banach separation theorems in operator algebras}
\author[I. Choi]{Ikhan Choi}
\address{Graduate School of Mathematical Sciences\\The University of Tokyo\\3-8-1 Komaba Meguro-ku Tokyo 153-8914, Japan}
\email{choi@ms.u-tokyo.ac.jp}

\begin{document}

\begin{abstract}
We affirmatively resolve a question posed by Uffe Haagerup in 1975 on the positive version of the bipolar theorem on the dual spaces of C$^*$-algebras.
As a direct consequence, we obtain a complete set of four positive Hahn-Banach separation theorems on von Neumann algebras, their preduals, C$^*$-algebras, and their duals.
Furthermore, with the idea used to solve the problem, we simplify Haagerup's original solution to Dixmier's problem on normal weights.
\end{abstract}

\maketitle

\section{Introduction}

In this paper, we prove the following theorem.
\begin{theorem*}
Let $M$ be a von Neumann algebra, and let $A$ be a C$^*$-algebra.
\begin{enumerate}
\item[(1)] If $F$ is a $\sigma$-weakly closed convex hereditary subset of $M^+$, then for any $x'\in M^+\setminus F$ there exists $\omega\in M_*^+$ such that $\omega(x')>1$ and $\omega(x)\le1$ for all $x\in F$.
\item[(2)] If $F_*$ is a norm closed convex hereditary subset of $M_*^+$, then for any $\omega'\in M_*^+\setminus F_*$ there exists $x\in M^+$ such that $\omega'(x)>1$ and $\omega(x)\le1$ for all $\omega\in F_*$.
\item[(3)] If $F$ is a norm closed convex hereditary subset of $A^+$, then for any $a'\in A^+\setminus F$ there exists $\omega\in A^{*+}$ such that $\omega(a')>1$ and $\omega(a)\le1$ for all $a\in F$.
\item[(4)] If $F^*$ is a weakly$^*$ closed convex hereditary subset of $A^{*+}$, then for any $\omega'\in A^{*+}\setminus F^*$ there exists $a\in A^+$ such that $\omega'(a)>1$ and $\omega(a)\le1$ for all $\omega\in F^*$.
\end{enumerate}
\end{theorem*}
\noindent Recall that a subset $F$ of the positive cone $E^+$ of a partially ordered real vector space $E$ is called \emph{hereditary} if $x\in F$ whenever $0\le x\le y$ in $E$ with $y\in F$, or equivalently $F=(F-E^+)^+$.
Here, we use $F_*$ and $F^*$ just for symbolic notations to emphasize that they are contained in $M_*$ and $A^*$ respectively, so they should not be perceived as a star operation of another set $F$.

The first three parts of the above theorem were originally proved by Uffe Haagerup in his master's thesis \cite{MR380438}.
In his monumental paper, he answered Dixmier's question \cite{MR59485} by proving that several conditions for normality of weights on a von Neumann algebra are in fact all equivalent, and posed three problems on normal weights, 1.10, 1.11, and 2.7.
Among these, the last one is the problem that asks whether the item (4) holds, and the solution to this problem is the main result of this paper.
The other two problems are also of independent interest and, as far as the author is aware, remain open.

The first part (1) plays a major role in the proof of that $\sigma$-weakly lower semi-continuous weight of a von Neumann algebra is given by the pointwise supremum of a set of positive normal linear functionals.
This is a second half of Dixmier's problem on normal weights.
The statements of the above collection of theorems in his original paper are written in terms of positive versions of the bipolar theorem $F^{r+r+}=F$ instead of the Hahn-Banach separation theorem as presented above, where the positive polar $F^{r+}$ can be defined as the positive part of the real polar 
\[F^{r+}:=F^r\cap E^{*+}=\{x^*\in E^{*+}:\sup_{x\in F}x^*(x)\le1\}\]
in an ordered real locally convex space $E$.
They are easily checked to be equivalent using the usual Hahn-Banach separation theorem.

Although the first three are already known results by Haagerup on the contrary to (4), we provide different proofs of them in order to motivate the idea of the proof of (4).
In more detail, when Haagerup proved (1), he heavily used the $\sigma$-strong topology and the strong continuity of continuous bounded functions, but such a nice dual topology for the dual pair of a von Neumann algebra and its $\sigma$-weak dual has no analogue in the dual of a C$^*$-algebra.
In this context, we provide a new proof of (1) using only the $\sigma$-weak topology, and extend the idea to prove (4) within the weak$^*$ topology.
This notably simplifies Haagerup's solution of the Dixmier problem by removing the usage of the strong operator topology from the proof of (1).
On the other hand, Haagerup used (1) to prove (2), but we will show that the Krein-\v Smulian theorem, which is essential in the proof of (1) and (4), is not required in (2) by proving it directly.
The part (3) is a straightforward corollary of (1), and it gives a different proof for the Combes theorem \cite{MR236721}, the C$^*$-algebraic counterpart of Dixmier's problem.
See Theorem 1.1.1~in \cite{MR4864224} for a C$^*$-algebraic proof of Combes' theorem.
Finally, combining these four results, we can also obtain the following correspondences.
\begin{corollary*}
For a von Neumann algebra $M$ and a C$^*$-algebra $A$, there are one-to-one correspondences
\[\begin{array}{ccccc}
\left\{\begin{tabular}{c}
normal\\subadditive\\weights on $M$
\end{tabular}\right\}&
\leftrightarrow&
\left\{\begin{tabular}{c}
$\sigma$-weakly closed\\convex hereditary\\subsets of $M^+$
\end{tabular}\right\}&
\leftrightarrow&
\left\{\begin{tabular}{c}
norm closed\\convex hereditary\\subsets of $M_*^+$
\end{tabular}\right\}
\end{array}\]
and
\[\begin{array}{ccccc}
\left\{\begin{tabular}{c}
lower semi-continuous\\subadditive\\weights on $A$
\end{tabular}\right\}&
\leftrightarrow&
\left\{\begin{tabular}{c}
norm closed\\convex hereditary\\subsets of $A^+$
\end{tabular}\right\}&
\leftrightarrow&
\left\{\begin{tabular}{c}
weakly$^*$ closed\\convex hereditary\\subsets of $A^{*+}$
\end{tabular}\right\}
\end{array}\]
\end{corollary*}

In Section 2, we introduce some preliminaries for the main proofs, most of which are well-known.
The notations introduced in Section 2 will be repeatedly used in Section 3, where we give the proofs of the four positive Hahn-Banach separation theorems.

\section{Preliminaries}

\subsection{Suppression by the one-parameter family of functional calculi}

The first step of the proof is the reformulation of the theorem into an inclusion problem.
This idea is due to Haagerup.
More precisely, we can transform each positive bipolar theorem equivalently to a statement of the form $(\overline{F-E^+})^+\subset F$.
In other words, when $x_i$ is a net in $F-E^+$ convergent to $x\ge0$ in a partially oredered real locally convex space $E$, there is a majorizing net $y_i\in F$ satisfying $x_i\le y_i$, we then need to find $y\in F$ such that $x\le y$ in order to deduce $x\in F$ with the hereditarity of $F$.
An obstacle of this argument is that a net $y_i$ has of course no limit points in general.
To address this, we suppress the net $y_i$ to make it bounded with the following one-parameter family of real functions.
\begin{definition}
For $\delta>0$, we define a function $f_\delta:(-\delta^{-1},\infty)\to\mathbb{R}$ such that
\[f_\delta(t):=t(1+\delta t)^{-1},\qquad t>-\delta^{-1}.\]
\end{definition}
\noindent It has many interesting properties such as operator monotonicity and concavity, strong convergence to the identity as $\delta\to0$, and the semi-group property in the sense that $f_\delta(f_{\delta'}(t))=f_{\delta+\delta'}(t)$ on a suitable domain of $t\in\mathbb{R}$.
We will only use the operator monotonicity and the boundedness from above, given each fixed $\delta>0$.
With these functions, for example in the situation of (1), we can suppress the net $y_i$ to define a bounded net $f_\delta(y_i)$ of positive elements.
This net is bounded above by $y_i\in F$ so that $f_\delta(y_i)\in F$ by the hereditarity of $F$, and we can take a $\sigma$-weakly convergent subnet to get $f_\delta(x)\in F$ from $f_\delta(x_i)\le f_\delta(y_i)\in F$ by the closedness and the hereditarity of $F$.
Then, the limit $\delta\to0$ and the closedness of $F$ will give $x\in F$.

However, when we try to apply this strategy, we encounter two technical issues.
One is that the net $x_i$ should be bounded for the net $f_\delta(x_i)$ to be well-defined for a sufficiently small but fixed $\delta>0$.
It can be overcome by Haagerup's trick involving the Krein-\v Smulian theorem.
The other is, unlike the $\sigma$-strong topology that Haagerup used, that we cannot expect the operation of taking functional calculus to be continuous with respect to the $\sigma$-weak topology.
For instance, when a self-adjoint net $x_i\in B(H)^{sa}$ satisfies an assumption $-(2\delta)^{-1}\le x_i$ for all $i$ and $x_i\to x$ $\sigma$-weakly, it does not necessarily follow $f_\delta(x_i)\to f_\delta(x)$ $\sigma$-weakly.
Our solution is to approximate the functions $f_\delta$ with affine functions.
The following lemma will be frequently used, and it allows us to approximate $f_\delta(x)$ directly with the $\sigma$-weakly convergent net $x_i$ instead of $f_\delta(x_i)$, while accompanied by small errors.
\begin{lemma}
Let $\delta>0$ and $t\in\mathbb{R}$.
\begin{enumerate}
\item[(1)] If $|t|\le2^{-1}\delta^{-\frac14}$ and $\delta\le1$, then $t\le f_\delta(t)+\delta^{\frac12}$.
\item[(2)] If $t\ge0$ and $c\ge0$, then $f_\delta(t+c)\le f_\delta(t)+c$.
\item[(3)] If $0\le t\le1$, then $(1+\delta)^{-1}t\le f_\delta(t)$.
\end{enumerate}
\end{lemma}
\begin{proof}
(1)
Since $\delta\le1<2^{\frac43}$ implies $t\ge-2^{-1}\delta^{-\frac14}>-\delta^{-1}$, the value $f_\delta(t)$ is well-defined.
The inequality is equivalent to $\delta^{\frac12}t^2-\delta t\le1$, and it holds for $|t|\le2^{-1}\delta^{-\frac14}$.

(2)
The inequality holds for $t=0$, so we can check by differentiation with respect to $t$ for fixed $c$.

(3)
It follows from $t(1+\delta t)\le t(1+\delta)$.
\end{proof}

\subsection{Bounded commutant Radon-Nikodym derivatives}

In the proofs of (2) and (4), we need to suppress linear functionals instead of operators, so we introduce the notion of Radon-Nikodym derivatives to convert functionals to operators as follows.
\begin{definition}
Let $M$ be a von Neumann algebra, and let $\psi\in M_*^+$.
Consider the Gelfand-Naimark-Segal representation $\pi:M\to B(H)$ associated to $\psi$, together with the canonical cyclic vector $\Omega\in H$.
Then, we have a positive bounded linear map $\theta:\pi(M)'\to M_*$ defined such that
\[\theta(h)(x):=\langle h\pi(x)\Omega,\Omega\rangle,\qquad h\in\pi(M)',\ x\in M.\]
Although it is not a standard terminology, we will call this linear map $\theta$ the (bounded) \emph{commutant Radon-Nikodym map} associated to $\psi$ in this paper.
\end{definition}
Let $\theta$ be the commutant Radon-Nikodym map associated to $\psi\in M_*^+$.
When $M$ is commutative so that $\pi(M)=\pi(M)'$, the inverse map $\theta^{-1}$ assigns to a linear functional an operator, which is exactly the Radon-Nikodym derivative with respect to $\psi$ in the classical measure-theoretic sense.
We include the adjective ``commutant'' to avoid confusion with other Radon-Nikodym derivatives in the theory of operator algebras such as the one of Takesaki-Pedersen or Connes' cocycle derivatives.
In the proof of (4), we will make use of this for the enveloping von Neumann algebra $A^{**}$ of a C$^*$-algebra $A$.

The image of $\theta$ associated to $\psi$ is described by
\[\operatorname{im}\theta=\{\omega\in M_*:\text{there is $C>0$ such that $|\omega(x)|\le C\psi(x)$ for all $x\in M^+$}\},\]
and the Radon-Nikodym derivative $\theta^{-1}(\omega)\in\pi(M)'$ of $\omega\in\operatorname{im}\theta$ satisfies the bound $\|\theta^{-1}(\omega)\|\le C$, where $C$ is a constant in the above description of the image of $\theta$.
In the proof of (2), for a sequence $\omega_n\in F_*-M_*^+$, a single reference functional $\psi\in M_*^+$ is constructed such that $|\omega_n(x)|\le C\psi(x)$ for all $n$ and $x\in M^+$, however in (4), we take $\psi_i\in A^{*+}$ depending on the net $\omega_i\in F^*-A^{*+}$ to have $|\omega_i(a)|\le C\psi_i(a)$ for all $a\in A^+$ and $i$.
The dependence of $\psi_i$ on $i$ is a key idea of the proof of (4), and we can check this dependence is reflected in the definitioin of the set $G^*$ in the proof.

Let $\omega\in M_*^{sa}$.
The Jordan decomposition theorem gives a unique pair $\omega_+,\omega_-\in M_*^+$ of positive normal linear functionals such that $\omega=\omega_+-\omega_-$ and $\|\omega\|=\|\omega_+\|+\|\omega_-\|$.
The absolute value of $\omega$ is defined as $[\omega]:=\omega_++\omega_-$.
Note that we always have $|\omega(x)|\le[\omega](x)$ for all $x\in M^+$, but unless $M$ is commutative, then for $\omega\in M_*^{sa}$ and $\psi\in M_*^+$ we cannot expect $[\omega]\le\psi$ when $|\omega(x)|\le\psi(x)$ for all $x\in M^+$.
Note also that when $A$ is a C$^*$-algebra and $\omega_i$ is a net in $A^{*sa}$, we have $\omega_{i+}\to0$ in norm if $\omega_i\to0$ in norm, but we do not have $\omega_{i+}\to0$ weakly$^*$ in general if $\omega_i\to0$ weakly$^*$ in $A^*$.
For further details, see IV.3 for the commutant Radon-Nikodym map and III.4 for the Jordan decomposition theorem in \cite{MR1873025}.

\section{Proofs of theorems}

\begin{theorem}\label{1}
Let $M$ be a von Neumann algebra, and consider the dual pair $(M^{sa},M_*^{sa})$.
If $F$ is a $\sigma$-weakly closed convex hereditary subset of $M^+$, then $F=F^{r+r+}$.
Equivalently, if $x'\in M^+\setminus F$, then there is $\omega\in M_*^+$ such that $\omega(x')>1$ and $\omega(x)\le1$ for $x\in F$.
\end{theorem}
\begin{proof}
Since the positive polar is represented by the real polar
\[F^{r+}=F^r\cap M_*^+=F^r\cap(-M^+)^r=(F\cup-M^+)^r=(F-M^+)^r,\]
the positive bipolar can be written as $F^{r+r+}=(F-M^+)^{rr+}=(\overline{F-M^+})^+$ by the usual real bipolar theorem, where the closure is for the $\sigma$-weak topology.
Because $F=(F-M^+)^+\subset(\overline{F-M^+})^+$, it suffices to prove the opposite inclusion $(\overline{F-M^+})^+\subset F$.

Define
\[G:=\left\{x\in M^{sa}:\begin{tabular}{c}
for any sufficiently small $\delta>0$\\
there is $y_\delta\in F$ such that\\
$\|y_\delta\|\le\delta^{-1}$ and $x\le y_\delta+\delta^{\frac12}$\\
\end{tabular}\right\}.\]
We prove $(\overline{F-M^+})^+\subset F$ via three steps, $F-M^+\subset G$, $G^+\subset F$, and $\overline G\subset G$.

Let $x\in F-M^+$, with $y\in F$ such that $x\le y$.
If we define $y_\delta:=f_\delta(y)$, then for sufficiently small $\delta>0$ such that $\|x\|\le2^{-1}\delta^{-\frac14}$ and $\delta\le1$, we have $x\le f_\delta(x)+\delta^{\frac12}\le y_\delta+\delta^{\frac12}$, so $x\in G$.

Let $x\in G^+$.
Take $y_\delta\in F$ such that $\|y_\delta\|\le\delta^{-1}$ and $x\le y_\delta+\delta^{\frac12}$ for sufficiently small $\delta$.
For any $\delta'>0$, we have
\[0\le(1+\delta'\|x\|)^{-1}x\le f_{\delta'}(x)\le f_{\delta'}(y_\delta+\delta^{\frac12})\le f_{\delta'}(y_\delta)+\delta^{\frac12}.\]
Since $f_{\delta'}(y_\delta)\in F$ is bounded for fixed $\delta'$, by considering the limit along a cofinal ultrafilter on the set of $\delta$, we have $(1+\delta'\|x\|)^{-1}x\in F$, so the limit $\delta'\to0$ gives $x\in F$.

Now it suffices to show $G$ is $\sigma$-weakly closed.
We claim that for any $r>0$ we have
\[\overline{(F-M^+)\cap M_{2r}}\subset G,\qquad G\cap M_r\subset\overline{(F-M^+)\cap M_{2r}},\]
where $M_r:=\{x\in M:\|x\|\le r\}$.
It these are true, then
\[G\cap M_r=\overline{(F-M^+)\cap M_{2r}}\cap M_r\]
is $\sigma$-weakly closed and convex in $M$ for all $r>0$, so the Krein-\v Smulian theorem proves that $G$ is $\sigma$-weakly closed.

Let $x_i\in (F-M^+)\cap M_{2r}$ be a net such that $x_i\to x$ $\sigma$-weakly in $M$ and $\|x_i\|\le2r$ for all $i$.
Take $y_i\in F$ such that $x_i\le y_i$ for all $i$.
For sufficiently small $\delta>0$ such that $2r\le2^{-1}\delta^{-\frac14}$ and $\delta\le1$, we have $x_i\le f_\delta(x_i)+\delta^{\frac12}\le f_\delta(y_i)+\delta^{\frac12}$.
By considering a cofinal ultrafilter on the index set of $i$, if we define $y_\delta$ as the $\sigma$-weak ultralimit of $f_\delta(y_i)$ of this ultrafilter, then we have $\|y_\delta\|\le\delta^{-1}$ and $x\le y_\delta+\delta^{\frac12}$, so $x\in G$.

Let $x\in G\cap M_r$.
Take $y_\delta\in F$ such that $\|y_\delta\|\le\delta^{-1}$ and $x\le y_\delta+\delta^{\frac12}$ for sufficiently small $\delta$.
If $\delta^{\frac12}<r$, then $x-\delta^{\frac12}\in(F-M^+)\cap M_{2r}$ $\sigma$-weakly converges to $x$ as $\delta\to0$, so $x\in\overline{(F-M^+)\cap M_{2r}}$.
This completes the proof.
\end{proof}

\begin{theorem}
Let $M$ be a von Neumann algebra, and consider the dual pair $(M_*^{sa},M^{sa})$.
If $F_*$ is a norm closed convex hereditary subset of $M_*^+$, then $F_*=F_*^{r+r+}$.
Equivalently, if $\omega'\in M_*^+\setminus F_*$, then there is $x\in M^+$ such that $\omega'(x)>1$ and $\omega(x)\le1$ for $\omega\in F_*$.
\end{theorem}
\begin{proof}
It is enough to prove $(\overline{F_*-M_*^+})^+\subset F_*$, where the closure is for the weak topology or equivalently for the norm by the convexity of $F_*-M_*^+$, so we begin our proof by fixing $\omega\in(\overline{F_*-M_*^+})^+$ and a sequence $\omega_n\in F_*-M_*^+$ such that $\omega_n\to\omega$ in norm of $M_*$.
Take $\varphi_n\in F_*$ such that $\omega_n\le\varphi_n$ for all $n$.
We may assume $\|\omega_n-\omega\|\le2^{-n}$ for all $n$ by passing to a subsequence.
Define
\[\psi:=\omega+\sum_n[\omega_n-\omega]+\sum_n2^{-n}\frac{\varphi_n}{1+\|\varphi_n\|}\]
in $M_*^+$, and let $\theta$ be the commutant Radon-Nikodym map associated to $\psi$.
Since $-\psi\le\omega_n\le\psi$ implies the boundedness $\|\theta^{-1}(\omega_n)\|\le1$ for all $n$, the weak convergence $\omega_n\to\omega$ in $M_*$ implies the convergence $\theta^{-1}(\omega_n)\to\theta^{-1}(\omega)$ in the weak operator topology of $\pi(M)'$.
By the Mazur lemma, we can take a net $\omega_i$ in the convex hull of $\omega_n$ such that $\theta^{-1}(\omega_i)\to\theta^{-1}(\omega)$ strongly in $\pi(M)'$, and the corresponding $\varphi_i\in F^*$ can be defined such that $\omega_i\le\varphi_i$ for all $i$.
In fact, the net $\omega_i$ can be taken to be a sequence because the commutant is $\sigma$-finite by the existence of the separating vector, but it is not necessary in here.
For each $i$ and $0<\delta<1$, define
\[\omega_\delta:=\theta(f_\delta(\theta^{-1}(\omega))),\qquad\omega_{i,\delta}:=\theta(f_\delta(\theta^{-1}(\omega_i))),\qquad\varphi_{i,\delta}:=\theta(f_\delta(\theta^{-1}(\varphi_i))),\]
where the functional calculus $f_\delta(\theta^{-1}(\omega_i))$ is well-defined because $-1\le\theta^{-1}(\omega_i)$ for all $i$.
Since $f_\delta(\theta^{-1}(\omega_i))\to f_\delta(\theta^{-1}(\omega))$ strongly in the commutant by the strong continuity of $f_\delta$, we have $\omega_{i,\delta}\to\omega_\delta$ weakly in $M_*$ for each $\delta$.
Taking a cofinal ultrafilter on the index set of $i$, define $\varphi_\delta$ by the image of the $\sigma$-weak ultralimit of $f_\delta(\theta^{-1}(\varphi_i))$ in the commutant under the Radon-Nikodym map $\theta$.
Then, $\varphi_{i,\delta}\to\varphi_\delta$ weakly in $M_*$ for each $\delta$.
Since $\varphi_i\in F^*$, the inequality $0\le\varphi_{i,\delta}\le\varphi_i$ implies $\varphi_{i,\delta}\in F_*$, and the weak convergence $\varphi_{i,\delta}\to\varphi_\delta$ in $M_*$ implies $\varphi_\delta\in F_*$.
Furthermore, $\omega_i\le\varphi_i$ implies $\omega_{i,\delta}\le\varphi_{i,\delta}$ and $0\le\omega_\delta\le\varphi_\delta$, so $\omega_\delta\in F_*$.
The weak convergence $\omega_\delta\to\omega$ in $M_*$ as $\delta\to0$ finally implies $\omega\in F_*$, so we get $(\overline{F_*-M_*^+})^+\subset F_*$.
\end{proof}

\begin{theorem}
Let $A$ be a C$^*$-algebra, and consider the dual pair $(A^{sa},A^{*sa})$.
If $F$ is a norm closed convex hereditary subset of $A^+$, then $F=F^{r+r+}$.
Equivalently, if $a'\in A^+\setminus F$, then there is $\omega\in A^{*+}$ such that $\omega(a')>1$ and $\omega(a)\le1$ for $a\in F$.
\end{theorem}
\begin{proof}
We directly prove the separation without invoking the arguments of positive bipolars.
Denote by $F^{**}$ the $\sigma$-weak closure of $F$ in the universal von Neumann algebra $A^{**}$.
We first show that $F^{**}$ is hereditary subset of $A^{**+}$.
Suppose $0\le x\le y$ in $A^{**}$ and $y\in F^{**}$.
Then, there is $z\in A^{**}$ such that $x^{\frac12}=zy^{\frac12}$.
Take bounded nets $b_i$ in $F$ and $c_i$ in $A$ such that $b_i\to y$ and $c_i\to z$ $\sigma$-strongly$^*$ in $A^{**}$ using the Kaplansky density theorem.
We may assume the indices of these two nets are shared by considering the product directed set.
Since both the multiplication and the involution of a von Neumann algebra on bounded parts are continuous in the $\sigma$-strong$^*$ topology, and since the square root on a positive bounded interval is strongly continuous, we have the $\sigma$-strong$^*$ limit
\[x=y^{\frac12}z^*zy^{\frac12}=\lim_ib_i^{\frac12}c_i^*c_ib_i^{\frac12},\]
so we obtain $x\in F^{**}$ from $b_i^{\frac12}c_i^*c_ib_i^{\frac12}\in F$.
Thus, $F^{**}$ is hereditary in $A^{**+}$.

Let $a'\in A^+\setminus F$.
If $a'\in F^{**}$, then we have a net $a_i$ in $F$ such that $a_i\to a'$ $\sigma$-weakly in $A^{**}$, which means that $a_i\to a'$ weakly in $A$, and by the weak closedness of $F$ in $A$ we get a contradiction $a'\in F^{**}\cap A=F$.
It implies $a'\in A^{**+}\setminus F^{**}$, so by Theorem \ref{1}, there is $\omega\in A^{*+}$ such that $\omega(a')>1$ and $\omega(a)\le1$ for all $a\in F\subset F^{**}$, and we are done.
\end{proof}

\begin{theorem}
Let $A$ be a C$^*$-algebra, and consider the dual pair $(A^{*sa},A^{sa})$.
If $F^*$ is a weakly$^*$ closed convex hereditary subset of $A^{*+}$, then $F^*=(F^*)^{r+r+}$.
Equivalently, if $\omega'\in A^{*+}\setminus F^*$, then there is $a\in A^+$ such that $\omega'(a)>1$ and $\omega(a)\le1$ for $\omega\in F^*$.
\end{theorem}
\begin{proof}
As above, we will prove $(\overline{F^*-A^{*+}})^+\subset F^*$.
Define
\[G^*:=\left\{\omega\in A^{*sa}:\begin{tabular}{c}
there is $\psi\in A^{*+}$, and there is $\varphi_\delta\in F^*$ \\
for any sufficiently small $\delta>0$,
such that\\
$\|\psi\|\le1$, $\|\varphi_\delta\|\le\delta^{-1}$, and $\omega\le\varphi_\delta+\delta^{\frac12}\psi$
\end{tabular}\right\}.\]
It suffices to show $F^*-A^{*+}\subset G^*$, $G^{*+}\subset F^*$, and $\overline{G^*}\subset G^*$.

Let $\omega\in F^*-A^{*+}$.
Take $\varphi\in F^*$ such that $\omega\le\varphi$.
Define, for $\delta>0$,
\[\psi:=\frac{[\omega]}{1+\|\omega\|}+\frac\varphi{(1+\|\omega\|)(1+\|\varphi\|)},\qquad\varphi_\delta:=\theta(f_\delta(\theta^{-1}(\varphi))),\]
where $\theta$ is the commutant Radon-Nikodym map associated to $\psi$.
The norm conditions $\|\psi\|\le1$ and $\|\varphi_\delta\|\le\delta^{-1}$ are easily checked.
For sufficiently small $\delta>0$ such that $\|\theta^{-1}(\omega)\|\le1+\|\omega\|\le2^{-1}\delta^{-\frac14}$ and $\delta\le1$, we have
\[\theta^{-1}(\omega)\le f_\delta(\theta^{-1}(\omega))+\delta^{\frac12}\le f_\delta(\theta^{-1}(\varphi))+\delta^{\frac12},\]
so $\omega\le\varphi_\delta+\delta^{\frac12}\psi$ and $\omega\in G^*$.

Let $\omega\in G^{*+}$.
Take $\psi\in A^{*+}$ and $\varphi_\delta\in F^*$ such that $\|\psi\|\le1$, $\|\varphi_\delta\|\le\delta^{-1}$, and $\omega\le\varphi_\delta+\delta^{\frac12}\psi$, for any sufficiently small $\delta>0$.
Let $\psi_\delta:=\omega+\delta\varphi+\psi$, and let $\theta_\delta$ be the associated commutant Radon-Nikodym map.
For any fixed $\delta'>0$, since $0\le\theta_\delta^{-1}(\omega)\le1$, we have
\begin{align*}
0&\le(1+\delta')^{-1}\theta_\delta^{-1}(\omega)\le f_{\delta'}(\theta_\delta^{-1}(\omega))\le f_{\delta'}(\theta_\delta^{-1}(\varphi_\delta+\delta^{\frac12}\psi))\\
&\le f_{\delta'}(\theta_\delta^{-1}(\varphi_\delta)+\delta^{\frac12})\le f_{\delta'}(\theta_\delta^{-1}(\varphi_\delta))+\delta^{\frac12},
\end{align*}
and it implies $0\le(1+\delta')^{-1}\omega\le\theta_\delta(f_{\delta'}(\theta_\delta^{-1}(\varphi_\delta)))+\delta^{\frac12}\psi_\delta$.
Since $\|\psi_\delta\|\le\|\omega\|+2$ is bounded and $\theta_\delta(f_{\delta'}(\theta_\delta^{-1}(\varphi_\delta)))\in F^*$ is also bounded for fixed $\delta'$, by considering the limit along a cofinal ultrafilter on the set of $\delta$, we have $(1+\delta')^{-1}\omega\in F^*$, so $\delta'\to0$ gives $\omega\in F^*$.

To show $G^*$ is weakly$^*$ closed, we claim for any $r>0$ that 
\[\overline{(F^*-A^{*+})\cap A_{2r}^*}\subset G^*,\qquad G^*\cap A_r^*\subset\overline{(F^*-A^{*+})\cap A_{2r}^*},\]
where $A_r^*:=\{\omega\in A^*:\|\omega\|\le r\}$.
If these are true, then
\[G^*\cap A_r^*=\overline{(F^*-A^{*+})\cap A_{2r}^*}\cap A_r^*\]
is weakly$^*$ closed and convex in $A^*$ for all $r>0$, so the Krein-\v Smulian theorem shows that $G^*$ is weakly$^*$ closed.

Let $\omega_i\in(F^*-A^{*+})\cap A_{2r}^*$ be a net such that $\omega_i\to\omega$ weakly$^*$ in $A^*$.
Following the proof of $F^*-A^{*+}\subset G^*$, we can take in the same way $\psi_i\in A^{*+}$ and $\varphi_{i,\delta}\in F^*$ such that $\|\psi_i\|\le1$, $\|\varphi_{i,\delta}\|\le\delta^{-1}$, and $\omega_i\le\varphi_{i,\delta}+\delta^{\frac12}\psi_i$, for uniformly sufficiently small $\delta$ such that $1+2r\le2^{-1}\delta^{-\frac14}$ because $\|\omega_i\|$ is bounded by $2r$.
Since the three conditions are preserved by the weak$^*$ convergence, taking the limit along a cofinal ultrafilter on the index set of $i$, we can obtain limit points $\psi$ and $\varphi_\delta$ so that $\omega\in G^*$.

Let $\omega\in G^*\cap A_r^*$.
Take $\psi\in A^{*+}$ and $\varphi_\delta\in F^*$ with $\|\psi\|\le1$, $\|\varphi_\delta\|\le\delta^{-1}$, and $\omega\le\varphi_\delta+\delta^{\frac12}\psi$, for any sufficiently small $\delta>0$.
If $\delta^{\frac12}<r$, then $\omega-\delta^{\frac12}\psi\in(F^*-A^{*+})\cap A_{2r}^*$ converges to $\omega$ weakly$^*$ in $A^*$ as $\delta\to0$, we have $\omega\in\overline{(F^*-A^{*+})\cap A_{2r}^*}$.
This completes the proof.
\end{proof}

Haagerup left in his paper \cite{MR380438} a comment without a proof that Theorem 3.4 is not difficult when a C$^*$-algebra is commutative.
As a final remark, we give a direct proof of this.

\begin{corollary}
Let $A$ be a commutative C$^*$-algebra, and consider the dual pair $(A^{*sa},A^{sa})$.
If $F^*$ is a weakly$^*$ closed convex hereditary subset of $A^{*+}$, then $F^*=(F^*)^{r+r+}$.
\end{corollary}
\begin{proof}
We first see that $\omega_+\le\omega'_+$ if $\omega\le\omega'$ in $A^{*sa}$.
Consider the ramp functions $t\mapsto t_+:=\max\{0,t\}$ and $t\mapsto t_-:=\max\{0,-t\}$ defined on $t\in\mathbb{R}$.
Let $\omega\in A^{*sa}$ and $\psi\in A^{*+}$ such that $\omega$ belongs to the image of the commutant Radon-Nikodym map $\theta$ associated to $\psi$.
Let $\pi:A\to B(H)$ be the Gelfand-Naimark-Segal representation associated to $\psi$ with the canonical cyclic vector $\Omega\in H$.
Let $p\in\pi(A)'$ be a projection such that $|\theta^{-1}(\omega)|=\theta^{-1}(\omega)p$.
Then, $p\in\pi(A)''$ because $\pi(A)''$ is a maximal commutative subalgebra of $B(H)$ by Theorem III.1.2 in \cite{MR1873025}, or in a straightforward way because we can identify $p\in\pi(A)'$ with a bounded linear functional $\pi(A)''_*\to\mathbb{C}:(x\mapsto\langle x\pi(a)\Omega,\pi(b)\Omega\rangle)\mapsto\langle p\pi(a)\Omega,\pi(b)\Omega\rangle$, where $x\in\pi(A)''$ and $a,b\in A$.
If we take $a\in A$ such that $\|\pi(a)\|\le1$ and $|\langle\theta^{-1}(\omega)(p-\pi(a))\Omega,\Omega\rangle|<\varepsilon$ for arbitrarily chosen $\varepsilon>0$ using the Kaplansky density, then the limit $\varepsilon\to0$ on
\begin{align*}
\|\theta(\theta^{-1}(\omega)_+)\|+\|\theta(\theta^{-1}(\omega)_-)\|
&=\langle|\theta^{-1}(\omega)|\Omega,\Omega\rangle
=\langle\theta^{-1}(\omega)p\Omega,\Omega\rangle\\
&\le|\langle\theta^{-1}(\omega)\pi(a)\Omega,\Omega\rangle|+\varepsilon\le\|\omega\|+\varepsilon
\end{align*}
shows that the positive part $\omega_+$ in the Jordan decomposition of $\omega$ is given by the functional calculus $\omega_+=\theta(\theta^{-1}(\omega)_+)$ by the ramp function.
Moreover, every monotone function can behave like an operator monotone function on a commutative C$^*$-algebra, so we can deduce $\omega_+\le\omega'_+$ if $\omega\le\omega'$ in $A^{*sa}$.

Now we prove the theorem.
Define
\[G^*:=\left\{\omega\in\overline{F^*-A^{*+}}:\begin{tabular}{c}there is a bounded net $\omega_i\in F^*-A^{*+}$\\such that $\omega_i\to\omega$ weakly$^*$ in $A^*$\end{tabular}\right\}.\]
It suffices to show $F^*-A^{*+}\subset G^*$, $G^{*+}\subset F^*$, and $\overline{G^*}\subset G^*$.
We can easily check $F^*-A^{*+}\subset G^*$ by considering constant sequences.

If $\omega\in G^{*+}$, then there exist a bounded net $\omega_i\in F^*-A^{*+}$ and a net $\varphi_i\in F^*$ such that $\omega_i\le\varphi_i$ for all $i$ and $\omega_i\to\omega$ weakly$^*$ in $A^*$ by definition of $G^*$.
Since $0\le\omega_{i+}\le\varphi_i\in F^*$ implies $\omega_{i+}\in F^*$, and since the net $\omega_{i+}$ is bounded so that we may assume $\omega_{i+}\to\omega'$ weakly$^*$ in $A^*$, we have $0\le\omega\le\omega'\in F^*$, whence $\omega\in F^*$.

To show $G^*$ is weakly$^*$ closed, take a bounded net $\omega_i\in G^*$ in the spirit of the Krein-\v Smulian theorem such that $\omega_i\to\omega$ weakly$^*$ in $A^*$.
Then, for each $i$ we have a bounded net $\omega_{ij}\in F^*-A^{*+}$ and a net $\varphi_{ij}\in F^*$ such that $\omega_{ij}\le\varphi_{ij}$ for all $j$ and $\omega_{ij}\to\omega_i$ weakly$^*$ in $A^*$ by definition of $G^*$.
Note that $\omega_{ij}\le\varphi_{ij}$ implies $0\le\omega_{ij+}\le\varphi_{ij}\in F^*$, and $\omega_{ij+}\in F^*$ is a bounded net for each $i$ so that we may assume $\omega_{ij+}\to\omega_i'$ weakly$^*$ in $A^*$ with $\omega_i'\in F^*$.
Again, $\omega_i\le\omega_i'$ implies $0\le\omega_{i+}\le\omega_i'\in F^*$, and $\omega_{i+}\in F^*$ is a bounded net so that we may assume $\omega_{i+}\to\omega'$ weakly$^*$ in $A^*$ with $\omega'\in F^*$.
Then, $\omega\le\omega'\in F^*$ implies that $\omega\in F^*-A^{*+}\subset G^*$.
\end{proof}

\section*{Acknowledgments}
The author is grateful to his advisor Yasuyuki Kawahigashi for his enormous support and encouragement during the master's studies.
The author would also like to thank Narutaka Ozawa for his fruitful comments, particularly pointing out a mistake of the author in the original version, and for extremely valuable suggestions on a simpler proof than the complicated original proof.
Finally the author thanks Hikaru Awazu and Ayoub Hafid for helpful discussions.

\bibliographystyle{alpha}
\bibliography{bib}

\begin{thebibliography}{Com68}

\bibitem[Com68]{MR236721}
Fran\c{c}ois Combes.
\newblock Poids sur une {$C\sp{\ast} $}-alg\`ebre.
\newblock {\em J. Math. Pures Appl. (9)}, 47:57--100, 1968.

\bibitem[Dix53]{MR59485}
J.~Dixmier.
\newblock Formes lin\'eaires sur un anneau d'op\'erateurs.
\newblock {\em Bull. Soc. Math. France}, 81:9--39, 1953.

\bibitem[Haa75]{MR380438}
Uffe Haagerup.
\newblock Normal weights on {$W\sp{\ast} $}-algebras.
\newblock {\em J. Functional Analysis}, 19:302--317, 1975.

\bibitem[Tak02]{MR1873025}
M.~Takesaki.
\newblock {\em Theory of operator algebras. {I}}, volume 124 of {\em
  Encyclopaedia of Mathematical Sciences}.
\newblock Springer-Verlag, Berlin, 2002.
\newblock Reprint of the first (1979) edition, Operator Algebras and
  Non-commutative Geometry, 5.

\bibitem[Tho24]{MR4864224}
Klaus~Erik Thomsen.
\newblock {\em An introduction to {KMS} weights}, volume 2362 of {\em Lecture
  Notes in Mathematics}.
\newblock Springer, Cham, [2024] \copyright 2024.

\end{thebibliography}

\end{document}